\documentclass[11pt,reqno]{amsart}

\usepackage{multirow}
\usepackage{amssymb}
\usepackage{amsthm}
\usepackage{amsfonts}
\usepackage{amsmath}
\usepackage{graphicx}
\usepackage[margin=1in]{geometry}
\usepackage[colorlinks=true,linkcolor=blue,filecolor=blue,citecolor=red]{hyperref}

\newcommand{\pq}[2]{\genfrac{[}{]}{0pt}{}{#1}{#2}}

\theoremstyle{plain}
\newtheorem{theorem}{Theorem}
\newtheorem{lemma}[theorem]{Lemma}
\newtheorem{corollary}[theorem]{Corollary}

\theoremstyle{definition}

\theoremstyle{remark}

\begin{document}

\title{\bf Generalized $q$-Stirling numbers and normal ordering}

\author[R.B.~Corcino]{Roberto B.~Corcino}
\address{Mathematics and ICT Department, Cebu Normal University, 6000 Cebu City, Philippines}
\email[R.B.~Corcino]{rcorcino@yahoo.com} 

\author[R.O.~Celeste]{Richell O.~Celeste}
\author[K.J.M.~Gonzales]{Ken Joffaniel M.~Gonzales}
\address{Institute of Mathematics, University of the Philippines Diliman, 1101 Quezon City, Philippines}
\email[R.O.~Celeste]{ching@math.upd.edu.ph}
\email[K.J.M.~Gonzales]{kmgonzales@upd.edu.ph}

\begin{abstract}
The normal ordering coefficients of strings consisting of $V,U$ which satisfy $UV=qVU+hV^s$ ($s\in\mathbb N$) are considered. These coefficients are studied in two contexts: first, as a multiple of a sequence satisfying a generalized recurrence, and second, as $q$-analogues of rook numbers under the row creation rule introduced by Goldman and Haglund. A number of properties are derived, including recurrences, expressions involving other $q$-analogues and explicit formulas. We also give a Dobinsky-type formula for the associated Bell numbers and the corresponding extension of Spivey's Bell number formula. The coefficients, viewed as rook numbers, are extended to the case $s\in\mathbb R$ via a modified rook model.\end{abstract}

\maketitle

\section{Introduction} Let $V$ and $U$ be operators (or variables) that satisfy the commutation relation $[U,V]=UV-VU=1$. Given a string $w$ consisting of $U$'s and $V$'s, the \emph{normally ordered} form of $w$ is an equivalent operator expressed as $\sum c_{i,j} V^i U^j$. The normally ordered form may be computed using the commutation relation alone, i.e., by replacing all occurrences of $UV$ with $VU+1$, a task which can be cumbersome especially for long strings. It has been shown, however, that the normal ordering coefficients are numbers studied in combinatorics. In quantum physics, for instance, one encounters the boson operators $a$ and $a^\dagger$ which satisfy $[a,a^\dagger] =aa^\dagger - a^\dagger a = 1$. It is known that $(a^\dagger a)^n = \sum_{k=0}^n S(n,k) {a^\dagger}^k a^k$, where the $S(n,k)$ are the Stirling numbers of the second kind. Combinatorially, the number $S(n,k)$ counts the number of partitions of $\{1,2,\ldots,n\}$ into $k$ disjoint, non-empty subsets. Since only the commutation relation is considered, it follows that any pair of operators that satisfies the same commutation relation will have the same coefficients. One such pair of operators is given by $X$ and $D$ which act on the monomial $x^n$ by $Vx^n=x^{n+1}$ and $Ux^n=nx^{n-1}$. Since $[U,V]=1$, then $(VU)^n = \sum_{k=0}^n S(n,k) V^k U^k$.

The connection between normal ordering and rook theory was demonstrated by Navon \cite{Nav} who proved that the normal ordering coefficients are given by rook numbers on a Ferrers board. Varvak \cite{Var} derived explicit formulas for these coefficients using rook factorization. Blasiak \cite{Bla}, and El-Desouky Caki\'c and Mansour \cite{DesCakMan} also computed explicit formulas using other methods. These coefficients may be considered as some form of generalized Stirling numbers.

Generalizations of the classical commutation relation have also been introduced and studied. Katriel and Kibler \cite{KatKib} considered the $q$-deformed commutation relation $[a,a^\dagger]_q = aa^\dagger - qa^\dagger a = 1$ and showed that the coefficients of $(a^\dagger a)^n$ are a $q$-analogue of the Stirling numbers of the second kind, i.e., $(a^\dagger a)^n = \sum_{k=0}^n S_q[n,k] {a^\dagger}^k a^k$. By letting $q=1$, we recover the non-deformed commutation relation and the Stirling numbers of the second kind. Mansour and Schork \cite{ManSch} extensively studied normal ordering relations of the form $xy=qyx+hf(y)$. For the case where $f(y)=y^s$, they derived the normal ordering of the expression $(x+y)^n$, thereby extending Newton's binomial formula. In \cite{ManSchSha1}, they focused on the coefficients of $(VU)^n$, where $V,U$ satisfy $UV - VU = hV^s$ and obtained properties of the corresponding generalization of Stirling numbers and Bell numbers. They continued the study of these numbers in \cite{ManSchSha2} where they obtained more properties and introduced a $q$-analogue through the relation $UV - qVU = hV^s$. (We also mention that the papers \cite{ManSch, ManSchSha1, ManSchSha2} also contain an excellent account of some literature on the subject.)

In this paper, we continue the study of normal ordering coefficients given the commutation relation $UV - qVU = hV^s$ where $s\in\mathbb N$. The outline of the paper is as follows. In Section \ref{first}, we study the properties of the coefficients of $(VU)^n$ in the context of a sequence satisfying a generalized recurrence relation. In Section \ref{second}, we obtain explicit formulas for the normal ordering coefficients of arbitrary strings $H_{\textnormal{\textbf{r}},\textnormal{\textbf{s}}}=V^{r_n}U^{s_n}\cdots V^{r_2}U^{s_2}V^{r_1}U^{s_1}$ using an interpretation of normal ordering in terms of rook placements. In Section \ref{third}, we resume to the coefficients of $(VU)^n$ and further exploit the rook model to obtain identities for the associated Bell numbers. Finally, in Section \ref{fourth}, we extend the coefficients to the case $s\in\mathbb R$ using a modified rook model.

As we shall be dealing with many forms of generalized Stirling numbers, we summarize some notation in the table below for the convenience of the reader. The presence of the parameters $\textbf r, \textbf s,h,s$ and $q$ means that they are assumed to be arbitrary. Square brackets are used to refer to $q$-analogues. A similar convention is applied to generalized Bell numbers.

\begin{table}[htbp]
\label{table}
\begin{center}
\begin{tabular}{|l|l|l|l|}
\hline \textbf{String} & \textbf{Parameters} & \textbf{Coefficients}
& \textbf{Original Notation}\\ \hline $(VU)^n$ & $q=1$, $h=1,s=0$ & $S(n,k)$ &\\

\hline $(VU)^n$ & $q$ arbitrary, $h=1,s=0$ & $S_q[n,k]$ &\\

\hline $(VU)^n$ & $q=1,h\neq 0,s$ arbitrary & $S_{s,h}(n,k)$ & $\mathfrak S_{s;h}(n,k)$ \cite{ManSchSha1,ManSchSha2}\\

\hline $(VU)^n$ & $q,h\neq 0,s$ arbitrary & $S_{s,h,q}[n,k]$ & $\mathfrak S_{s;h}(n,k|q)$ \cite{ManSchSha2} \\

\hline $H_{\textbf r, \textbf s}$ & $q=1,h=1,s=0$ & $S^{\textbf r,\textbf s}(k)$ & $S_{\textbf r,\textbf s}(k)$ \cite{Bla} \\ \hline $H_{\textbf r, \textbf s}$ & $q,h\neq 0,s$ arbitrary & $S^{\textbf r,\textbf s}_{s,h,q}[k]$ & \\ \hline \end{tabular} \end{center} \end{table}

\section{The numbers $S_{s,h,q}[n,k]$}\label{first}

For $x\in \mathbb R$, we define the $q$-analogue of $x$ by $[x]_q=\frac{q^x-1}{q-1}$ for $x\neq 0$ and $[0]_q=0$. As $q\rightarrow 1^{-}$, $[x]_q \rightarrow x$. If $x\in \mathbb N$, then $[x]_q = 1+q+\cdots+q^{x-1}$ so that the replacement $q=1$ is sufficient to recover $x$.

Let $h\in\mathbb C \backslash\{0\}$ and $s\in \mathbb N$. For $U,V$ that satisfy the commutation relation $UV-qVU=hV^s$. Mansour, Schork and Shattuck \cite{ManSchSha1} showed that the coefficients $S_{s,h,q}[n,k]$ in
\begin{equation*}
(VU)^n = \sum_{k=0}^n S_{s,h,q}[n,k] V^{sn-(s-1)k} U^k 
\end{equation*} 
satisfy the recurrence relation
\begin{equation}\label{recrel}
S_{s,h,q}[n,k] = q^{s(n-1)-(s-1)(k-1)}S_{s,h,q}[n-1,k-1] + h[s(n-1)-(s-1)k]_qS_{s,h,q}[n-1,k]\,,
\end{equation}
with initial conditions $S_{s,h,q}[n,0]=S_{s,h,q}[0,n]=\delta_{0,n}$. If $q=1$, we obtain the commutation relation $UV-VU=hV^s$ and the numbers $S_{s,h}(n,k)$ which were studied in \cite{ManSchSha1}. Mansour, Schork and Shattuck \cite{ManSchSha2} also showed that the number $S_{s,h,q}[n,k]$ can be expressed in terms of $b_{n,k} = q^{(s-1)\binom{k}{2}-s\binom{n}{2}} S_{s,h,q}[n,k] $ where $b_{n,k}$ satisfies
\begin{equation}\label{bnkrec}
b_{n,k} = b_{n-1,k-1} + (h/q) ([s(n-1)]_{1/q} - [(s-1)k]_{1/q}) b_{n-1,k}\,.
\end{equation}

The recurrence \eqref{bnkrec} is useful in deriving some basic properties. However, we will find it more convenient to work on a more general setting. Specifically, let $\textbf v=(v_0,v_1,\ldots)$ and $\textbf w=(w_0,w_1,\ldots)$ be sequences from a ring $K$ and define the numbers
$A^{\textnormal{\textbf {v}}, \textnormal{\textbf{w}}}_{n,k}$ by the recurrence relation
\begin{equation}\label{ankrec}
A^{\textnormal{\textbf {v}}, \textnormal{\textbf{w}}}_{n,k} = A^{\textnormal{\textbf {v}}, \textnormal{\textbf{w}}}_{n-1,k-1} + (v_{n-1} + w_{k}) A^{\textnormal{\textbf {v}}, \textnormal{\textbf{w}}}_{n-1,k}\,,
\end{equation}
with initial conditions $A^{\textnormal{\textbf {v}}, \textnormal{\textbf{w}}}_{0,n}=\delta_{0,n}$ and $A^{\textnormal{\textbf {v}}, \textnormal{\textbf{w}}}_{n,0} = (v_{n-1} + w_0) (v_{n-2} + w_0) \cdots (v_{0} + w_0)$. We also call $\textbf v$ and $\textbf w$ weight functions. When only the value of a weight function at $i=0,1,\ldots$ is specified (for instance, $v_i$), it is understood that the corresponding weight function is the same letter in boldface without the subscripts. If $v_i=[si]_{1/q},w_i=-[(s-1)i]_{1/q}$, then by \eqref{bnkrec}, $A^{\textnormal{\textbf {v}}, \textnormal{\textbf{w}}}_{n,k} = (h/q)^{-(n-k)} q^{(s-1)\binom{k}{2}-s\binom{n}{2}} S_{s,h,q}[n,k]$.

We mention that for linear weight functions, the recurrence relation \eqref{ankrec} has been considered by Xu \cite{Xu}, Hsu and Shuie \cite{HsuShu} and El-Desouky and Caki\'c \cite{DesCak}. Some $q$- and $p,q$-analogues have also been studied by Corcino, Hsu and Tan \cite{CorHsuTan}, and Remmel and Wachs \cite{RemWac} (although we note that the type-II $p,q$-analogue in \cite{RemWac} does not seem to fall as $A^{\textnormal{\textbf {v}}, \textnormal{\textbf{w}}}_{n,k}$). An explicit formula was obtained in \cite{ManMulSha} and \cite{Xu}, which is given by
\begin{equation*}
A^{\textnormal{\textbf {v}}, \textnormal{\textbf{w}}}_{n,k} 
= \sum_{j=0}^k \frac{\prod_{i=0}^{n-1} (w_j + v_i)}{\prod_{\substack{i=0,\neq j}}^k (w_j-w_i)}\,,
\end{equation*}
provided that the $w_i$'s are distinct. For $s\neq 1$, the formula above gives us (see \cite[Theorem 45]{ManSchSha2})
\begin{equation*} S_{s,h,q}[n,k] = h^{n-k}q^{s\binom{n}{2}-(s-1)\binom{k}{2}-(n-k)} \sum_{j=0}^k \frac{\prod_{i=0}^{n-1} ([si]_{1/q} - [(s-1)j]_{1/q})}{\prod_{\substack{i=0,\neq j}}^k ([(s-1)i]_{}-[(s-1)j]_{1/q})}\,.
\end{equation*}


The theorem that follows gives different formulations for the numbers $A^{\textnormal{\textbf {v}}, \textnormal{\textbf{w}}}_{n,k}$ in terms of expressions that are analogous to elementary and complete symmetric functions.

\begin{theorem}\label{symfcn} The following identities hold: 
\begin{align}
A^{\textnormal{\textbf {v}}, \textnormal{\textbf{w}}}_{n,k} &=\sum_{0\leq i_1 < i_2 < ... < i_{n-k} \leq n-1}~\prod_{j=1}^{n-k} (v_{i_j} + w_{i_j-j+1}) \label{tdset}\\
A^{\textnormal{\textbf {v}}, \textnormal{\textbf{w}}}_{n,k} &= \sum_{0\leq i_1 \leq i_2 \leq ... \leq i_{n-k} \leq k}~\prod_{j=1}^{n-k} (v_{i_j+j-1} + w_{i_j}) \label{tset}\\
A^{\textnormal{\textbf {v}}, \textnormal{\textbf{w}}}_{n,k} &=\sum_{i_0+ i_1 + i_2 + ... + i_k = n-k} ~\prod_{j=0}^{k}~\prod_{l=0}^{i_j-1} (v_{j + l + i_0 + i_1+i_2+ ... +i_{j-1}} + w_j) \,.\label{tsetalt}
\end{align}
\begin{proof} It can be shown by partitioning the set of indices that the $RHS$ of \eqref{tdset} and \eqref{tset} satisfy the recursion \eqref{ankrec}. Identity \eqref{tsetalt} is a restatement of \eqref{tset}. \end{proof} \end{theorem}
 
Let $H$ be a set. The $m$-th \emph{elementary symmetric function} $e$ and $m$-th \emph{complete symmetric function} $h$ are defined as follows: $e_{m}(H)$ (resp. $h_{m}(H)$) is the sum of all products of $m$ elements from $H$ taken without (resp. with) replacement. Denote by $\textbf 0$ the function that is identically 0. Observe that by \eqref{tdset}, $A^{\textbf {v}, \textbf{0}}_{n,k}=e_{n-k}(\{v_0,v_1,\ldots,v_{n-1}\})$ and by \eqref{tset} or \eqref{tsetalt}, $A^{\textbf {0}, \textbf{v}}_{n,k}=h_{n-k}(\{v_0,v_1,\ldots,v_{k}\})$.

In the next theorem, we denote the matrix whose $(n,k)$-th entry is $A^{\textnormal{\textbf {v}}, \textnormal{\textbf{w}}}_{n,k}$ by $\left[A^{\textnormal{\textbf {v}}, \textnormal{\textbf{w}}}_{n,k}\right]$.

\begin{theorem}\label{big1} The following identities hold:
\begin{enumerate}

\item A matrix factorization
\begin{align}
\left[ A^{\textnormal{\textbf {v}}, \textnormal{\textbf{w}}}_{n,k} \right] = \left[
A^{\textnormal{\textbf {v}}, \textnormal{\textbf{0}}}_{n,k} \right] \left[ A^{\textnormal{\textbf {0}}, \textnormal{\textbf{w}}}_{n,k}
\right]\mbox{, or equivalently, }~~A^{\textnormal{\textbf {v}}, \textnormal{\textbf{w}}}_{n,k} &= \sum_{j=k}^{n} A^{\textnormal{\textbf {v}},
\textnormal{\textbf{0}}}_{n,j} A^{\textnormal{\textbf {0}}, \textnormal{\textbf{w}}}_{j,k} \label{matfac}
\end{align}

\item Orthogonality relation
\begin{align}
\sum_{k=m}^n A^{\textnormal{\textbf {v}}, \textnormal{\textbf{w}}}_{n,k} A^{\textnormal{\textbf {-w}},
\textnormal{\textbf{-v}}}_{k,m} &= \delta_{n,m}\label{orth}
\end{align}

\item A pair of inverse relations
\begin{align}
(x+v_0)(x+v_1) \cdots (x+v_{n-1}) &= \sum_{k=0}^n A^{\textnormal{\textbf {v}}, \textnormal{\textbf{w}}}_{n,k} (x-w_0)(x-w_1) \cdots (x-w_{k-1}) \label{inv1}\\
(x-w_0)(x-w_1) \cdots (x-w_{n-1}) &= \sum_{k=0}^n A^{\textnormal{\textbf {-w}}, \textnormal{\textbf{-v}}}_{n,k} (x+v_0)(x+v_1) \cdots
(x+v_{k-1})\label{inv2}
\end{align}

\item Convolution formula
\begin{align} A^{\textnormal{\textbf {v}}, \textnormal{\textbf{w}}}_{l+m,n} &=
\sum_{k=0}^n A^{\textnormal{\textbf {v}}, \textnormal{\textbf{w}}}_{l,k} A^{\textnormal{\textbf{v}}_{+l},
\textnormal{\textbf{w}}_{+k}}_{m,n-k}\,, \label{conv}\mbox{~where for any weight function $\textnormal{\textbf{f}}$, we define
$\textnormal{\textbf{f}}_{+m}=(f_m,f_{m+1},\ldots)$}\,.
\end{align}
\end{enumerate}
\begin{proof} We prove the second equation in \eqref{matfac} by showing that the quantity on the $RHS$ satisfies the recurrence relation \eqref{ankrec}. Let $k\leq j\leq n$. By \eqref{ankrec},
\begin{align*}
A^{\textbf {v}, \textbf{0}}_{n,j}A^{\textbf {0}, \textbf{w}}_{j,k} &= (A^{\textbf {v}, \textbf{0}}_{n-1,j-1} + v_{n-1} A^{\textbf{v}, \textbf{0}}_{n-1,j}) (A^{\textbf {0}, \textbf{w}}_{j-1,k-1} + w_{k} A^{\textbf {0}, \textbf{w}}_{j-1,k}) \\
&= A^{\textbf {v},\textbf{0}}_{n-1,j-1} A^{\textbf {0}, \textbf{w}}_{j-1,k-1} + v_{n-1} A^{\textbf {v}, \textbf{0}}_{n-1,j} A^{\textbf {0}, \textbf{w}}_{j,k} + w_k A^{\textbf {v}, \textbf{0}}_{n-1,j-1} A^{\textbf {0}, \textbf{w}}_{j-1,k}
\end{align*}
Substituting this into \eqref{matfac} and collecting the necessary terms proves the identity.

For the orthogonality relation \eqref{orth}, we use $\left[ A^{\textbf {0}, \textbf{v}}_{n,k}\right]^{-1} = \left[ (-1)^{n-k} A^{\textbf {v}, \textbf{0}}_{n,k}\right] = \left[ A^{\textbf {-v}, \textbf{0}}_{n,k}\right]$ (see \cite[Identity (2.10)]{MedLer2}). By \eqref{matfac}, $\left[A^{\textnormal{\textbf {v}}, \textnormal{\textbf{w}}}_{n,k}\right]^{-1} = \left[ A^{\textbf {0}, \textbf{w}}_{n,k}\right]^{-1} \left[ A^{\textbf{v}, \textbf{0}}_{n,k}\right]^{-1} = \left[ A^{\textbf {-w}, \textbf{0}}_{n,k}\right] \left[ A^{\textbf {0}, \textbf{-v}}_{n,k}\right] = \left[A^{\textbf {-w}, \textbf{-v}}_{n,k}\right]$.

An inductive argument and the recursion \eqref{ankrec} proves \eqref{inv1}. Identity \eqref{orth} then establishes the equivalence of \eqref{inv1} with \eqref{inv2}. For \eqref{conv}, we use \eqref{tdset} and break each $(l+m-n)$-tuple of indices into two, namely, into those that satisfy $0\leq i_1 < \ldots <i_{l-k}\leq l-1$ and $l\leq i_{l-k+1} < \ldots < i_{l+m-n}\leq m+l-1$, for some $k$. This $k$ is unique since each coordinate in every $(l+m-n)$-tuple is unique. The first set of indices will produce $A^{\textnormal{\textbf {v}}, \textnormal{\textbf{w}}}_{l,k}$. On the other hand, the second set will produce 
\begin{align*}
&\sum_{l\leq i_{l-k+1} < \ldots < i_{l+m-n}\leq m+l-1}~\prod_{j=l-k+1}^{l+m-n} (v_{i_j}+w_{i_j - (j-1)})\\
&=\sum_{0\leq i_{1} < \ldots < i_{m-(n-k)}\leq m-1} \prod_{j=1}^{m-(n-k)} (v_{i_j+l}+w_{i_j - (j-1)+k})\\
&=A^{\textbf v_{+l}, \textbf w_{+k}}_{m,n-k}\,.
\end{align*}

All identities have now been proved. \end{proof} \end{theorem}

By Identity \eqref{matfac}, we can write $S_{s,h,q}[n,k]$ in terms of other $q$-Stirling numbers. Note
that the second kind of this analogue is different from the $q$-analogue $S_{0,1,q}[n,k]=S_q[n,k]$. By \eqref{recrel}, the numbers $S_q[n,k]$ satisfy the recurrence
\begin{equation*}
S_q[n,k] = q^{k-1} S_q[n-1,k-1] + [k]_q S_q[n-1,k]\,.
\end{equation*}
On the other hand, the other $q$-analogues satisfy
\begin{align*}
\hat{S}_q[n,k] &= \hat{S}_q[n-1,k-1] + [k]_q \hat{S}_q[n-1,k]\\
\hat{c}_q[n,k] &= \hat{c}_q[n-1,k-1] + [n-1]_q \hat{c}_q[n-1,k]\,.
\end{align*}
with $\hat{S}_q[n,0]=\hat{S}_q[0,n]=\hat{c}_q[n,0]=\hat{c}_q[0,n]=\delta_{n,0}$. Using these recurrence relations, we see that when $v_i=[i]_q$, $A^{\textbf {v}, \textbf{0}}_{n,k}=\hat{c}_q[n,k]$ and $A^{\textbf {0}, \textbf{v}}_{n,k}=\hat{S}_q[n,k]$. These analogues have been studied extensively by M\'edicis and Leroux \cite{MedLer1}.

\begin{corollary} The numbers $S_{s,h,q}[n,k]$ satisfy
\begin{equation}\label{exp1}
S_{s,h,q}[n,k] = h^{(n-k)} q^{s\binom{n}{2}-(s-1)\binom{k}{2}-n+k} \sum_{j=k}^n [s]_{1/q}^{n-j} (-[s-1]_{1/q})^{j-k} \hat{c}_{1/{q^{s}}}[n,j] \hat{S}_{1/{q^{s-1}}}[j,k]\,.
\end{equation}
Consequently, the numbers $S_{s,h}(n,k)$ are given by
\begin{equation*}
S_{s,h}(n,k) = h^{n-k} \sum_{j=k}^n s^{n-j} (1-s)^{j-k} c(n,j) S(j,k)\,.
\end{equation*}
where $c(n,k)$ and $S(n,k)$ are the (classical) Stirling numbers of the first kind and second kind, respectively.
\begin{proof} The result follows from \eqref{matfac} and the property $[si]_{1/q}=[s]_{1/q}[i]_{1/q^s}$, $[(s-1)i]_{1/q}=[s-1]_{1/q}[i]_{1/q^{s-1}}$. \end{proof} \end{corollary}

Define $c_{s,h,q}[n,k]$ by the recursion
\begin{equation*} c_{s,h,q}[n,k] = q^{(s-1)(n-1)-s(k-1)} c_{s,h,q}[n-1,k-1] + h [(s-1)(n-1)-sk]_q
c_{s,h,q}[n-1,k]\,,
\end{equation*}
with initial conditions $c_{s,h,q}[n,0]=c_{s,h,q}[0,n]=\delta_{0,n}$. The orthogonality relation for $S_{s,h,q}[n,k]$ involves $c_{s,h,q}[n,k]$ and is given in the next corollary.

\begin{corollary} For $m\leq n$, we have the following orthogonality relations
\begin{align*}
\sum_{k=m}^n S_{s,h,q}[n,k] c_{s,h,q}[k,m] &= \delta_{m,n} \\
\sum_{k=m}^n c_{s,h,q}[n,k]S_{s,h,q}[k,m] &= \delta_{m,n}\,.
\end{align*}
\begin{proof} Let $v_i=[si]_{1/q},w_i=-[(s-1)i]_{1/q}$. One can verify that $A^{\textbf {-w},
\textbf{-v}}_{n,k}=(h/q)^{-(n-k)}q^{s\binom{k}{2}-(s-1)\binom{n}{2}} c_{s,h,q}[n,k]$. The orthogonality relation \eqref{orth} implies that
\begin{equation*}
\sum_{k=m}^n (h/q)^{-(n-k)} q^{(s-1)\binom{k}{2}-s\binom{n}{2}} S_{s,h,q}[n,k] (h/q)^{-(k-m)}
q^{s\binom{m}{2}-(s-1)\binom{k}{2}} c_{s,h,q}[k,m] = \delta_{m,n}\,.
\end{equation*}
After clearing the powers of $h$ and $q$ we obtain the desired result. \end{proof} \end{corollary}

The next theorem gives some recurrence relations for $A^{\textnormal{\textbf {v}}, \textnormal{\textbf{w}}}_{n,k}$. Identities \eqref{rec1} and \eqref{rec3} reduce into $q$-analogues of \cite[Theorems 5.5 and 5.4]{ManSchSha1}, respectively, which we state in the following corollary.

\begin{theorem}\label{big2} The following recurrence relations for $A^{\textnormal{\textbf {v}}, \textnormal{\textbf{w}}}_{n,k}$ hold
\begin{align}
A^{\textnormal{\textbf {v}}, \textnormal{\textbf{w}}}_{n,k} &= \sum_{j=k}^n A^{\textnormal{\textbf {v}}, \textnormal{\textbf{w}}}_{j-1,k-1} \prod_{i=j}^{n-1} (v_i + w_{k}) \label{rec1} \\ A^{\textnormal{\textbf {v}}, \textnormal{\textbf{w}}}_{n,k} &= \sum_{j=k}^n (-1)^{j-k} A^{\textnormal{\textbf {v}}, \textnormal{\textbf{w}}}_{n+1,j+1} \prod_{i=k+1}^j (v_n+w_i) \label{rec2}\\
A^{\textnormal{\textbf {v}}, \textnormal{\textbf{w}}}_{n,k} &= \sum_{j=0}^k (v_{n-j-1}+w_{k-j})A^{\textnormal{\textbf {v}},
\textnormal{\textbf{w}}}_{n-j-1,k-j}\label{rec3}\,.
\end{align}
\begin{proof} The proof of the three identities uses \eqref{tdset}.

For \eqref{rec1}, we use the fact that for every $(n-k)$-tuple of indices $(i_1,\ldots,i_{n-k})$, there exists a unique $j$ satisfying $k\leq j\leq n$ such that $(i_{j-k+1},i_{j-k+2}\ldots,i_{n-k})=(j,j+1,\ldots,n-1)$ and $(i_1,\ldots,i_{j-k})$ satisfies $0\leq i_1 < \ldots < i_{j-k} \leq j-2$. A similar approach proves \eqref{rec3}.

To prove identity \eqref{rec2}, let $I_j$ be the set of $(n-j)$-tuples of indices satisfying $0\leq i_1 < \ldots < i_{n-j} \leq n$. Let
$I'_j=I_j \times (\underbrace{n,\ldots,n}_{j-k})$ and $\mathcal{I}$ be the multiset consisting of the union of all $I'_j$'s, $k\leq j\leq n$. Assign the weight $(-1)^{j-k}$ to each element of $\mathcal {I}$ belonging to $I'_j$. Then each $(n-k)$-tuple in $\mathcal {I}$ which does not satisfy $0\leq i_1 < \ldots < i_{n-k} \leq n-1$ occurs twice but with weights that are negatives of each other. \end{proof} \end{theorem}

\begin{corollary} The following recurrence relations for $S_{s,h,q}[n,k]$ hold
\begin{align}
S_{s,h,q}[n,k] &= \sum_{j=k}^n h^{n-j} q^{s(j-1)-(s-1)(k - 1)} S_{s,h,q}[j-1,k-1] \prod_{i=j}^{n-1} [si-(s-1)k]_q\label{recx}\\
S_{s,h,q}[n,k] &= \sum_{j=k}^n (-h)^{j-k} q^{sn(k-j-1)+(s-1)\binom{j+1}{2}-(s-1)\binom{k}{2}} S_{s,h,q}[n+1,j+1] \prod_{i=k+1}^j [sn-(s-1)i]_q\label{recy}\\
S_{s,h,q}[n,k] &= \sum_{j=0}^k h q^{j(sn-(s-1)k)-\binom{j+1}{2}}[s(n-j-1)-(s-1)(k-j)]_q S_{s,h,q}[n-j-1,k-j]\,.\label{recz}
\end{align}
\begin{proof} We use Theorem \ref{big2} with $v_i=[si]_{1/q}$ and $w_i=-[(s-1)i]_{1/q}$, and the observation that for any integers $b,c$, we have $[b-c]_q=q^{b-1}([b]_{1/q}-[c]_{1/q})$. \end{proof} \end{corollary}

Another identity of interest is the one by Carlitz (see \cite[Identities (4.5) and (4.7)]{MedLer1}) which gives the following equivalent
relations between $q$-binomial coefficients and $q$-Stirling numbers of the second kind
\begin{align}
\pq{n}{k}_q &= \sum_{j=k}^n \binom{n}{j} (q-1)^{j-k} \hat{S}_q[j,k]\,.\label{carlitz}\\
(q-1)^{n-k} \hat{S}_q[n,k] &= \sum_{j=k}^n (-1)^{n-j} \binom{n}{j} \pq{j}{k}_q\,.\nonumber
\end{align}
Here, $\pq{n}{k}_q$ denotes the $q$-binomial coefficients which is defined by $\pq{n}{k}_q=\frac{[n]_q!}{[k]_q![n-k]_q!}$, where for
$n\in\mathbb N$, $[n]_q!=[n]_q[n-1]_q\cdots [2]_q[1]_q$. For generalizations of Carlitz's identity, see \cite[Theorem 2.2]{MedLer2} where they are proved by distributing weights on certain tableaux.

Identity \eqref{carlitz} can be directly derived from $A^{\textnormal{\textbf {v}}, \textnormal{\textbf{w}}}_{n,k}$ using \eqref{matfac}. Let $f_i=q^i$ and $v_i=1$ and $w_i=q^i-1$. Then, $A^{\textnormal{\textbf {0}}, \textnormal{\textbf{f}}}_{n,k} = A^{\textnormal{\textbf {v}}, \textnormal{\textbf{w}}}_{n,k} = \sum_{j=k}^n A^{\textnormal{\textbf {1}}, \textnormal{\textbf{0}}}_{n,j} A^{\textnormal{\textbf {0}}, \textnormal{\textbf{w}}}_{j,k}$. It suffices to note that $A^{\textnormal{\textbf {0}}, \textnormal{\textbf{f}}}_{n,k}=\pq{n}{k}_q$, $A^{\textnormal{\textbf {1}}, \textnormal{\textbf{0}}}_{n,j}=\binom{n}{j}$ and (using $q^i-1=(q-1)[i]_q$), $A^{\textnormal{\textbf {0}}, \textnormal{\textbf{w}}}_{j,k}=(q-1)^{j-k}\hat{S}_q[j,k]$ by the symmetric functions forms of these numbers.

The next theorem gives a generalization of Carlitz's identity. Equivalent identities may be obtained by manipulating the matrices involved using the matrix formulation of identities \eqref{matfac} and \eqref{orth}.

\begin{theorem}\label{carltheorem} Let $c$ and $d$ be constants, and $\textnormal{\textbf v}^*, \textnormal{\textbf w}^*, \textnormal{\textbf v}, \textnormal{\textbf w}$ be weight functions such that $v_i=c+v^*_i,w_i=d+w^*_i$. Then,
\begin{align}
A^{\textnormal{\textbf {v}}, \textnormal{\textbf{w}}}_{n,k} &= \sum_{k\leq t_1 \leq t_2 \leq t_3 \leq n} \binom{t_3}{t_2}\binom{t_2}{t_1} c^{t_3-t_2} d^{t_2-t_1} A^{\textnormal{\textbf {v}}^*, \textnormal{\textbf{0}}}_{n,t_3}
A^{\textnormal{\textbf {0}}, \textnormal{\textbf{w}}^*}_{t_1,k} \label{carl1}\\
A^{\textnormal{\textbf {v}}^*,\textnormal{ \textbf{w}}^*}_{n,k} &= \sum_{k\leq t_1 \leq t_2 \leq t_3 \leq n} \binom{t_3}{t_2}\binom{t_2}{t_1} (-c)^{t_3-t_2} (-d)^{t_2-t_1} A^{\textnormal{\textbf {v}}, \textnormal{\textbf{0}}}_{n,t_3} A^{\textnormal{\textbf {0}}, \textnormal{\textbf{w}}}_{t_1,k}\,.\label{carl2}
\end{align}
\begin{proof} Let $\textbf c$ and $\textbf d$ denote the constant functions equal to $c$ and $d$, respectively. By repeated application of \eqref{matfac}, $\left[A^{\textnormal{\textbf {v}}, \textnormal{\textbf{w}}}_{n,k} \right]=\left[A^{\textnormal{\textbf {v}},
\textnormal{\textbf{0}}}_{n,k}\right]\left[A^{\textnormal{\textbf {0}}, \textnormal{\textbf{w}}}_{n,k} \right]=\left[A^{\textnormal{\textbf{v}}^*, \textnormal{\textbf{c}}}_{n,k} \right]\left[A^{\textnormal{\textbf {d}},
\textnormal{\textbf{w}}^*}_{n,k}\right]=\left[A^{\textnormal{\textbf {v}}^*, \textnormal{\textbf{0}}}_{n,k}\right]
\left[c^{n-k}\binom{n}{k}\right]\left[d^{n-k}\binom{n}{k}\right]\left[A^{\textnormal{\textbf {0}}, \textnormal{\textbf{w}}^*}_{n,k}\right]$, which proves \eqref{carl1}. Identity \eqref{carl2} is similarly proved. \end{proof}\end{theorem}

As a corollary, we obtain the following identity which expresses $S_{s,h,q}[n,k]$ in terms of $q$-binomial coefficients.

\begin{corollary} Let $q\neq 1$. The numbers $S_{s,h,q}[n,k]$ may be written as
\begin{multline*}\label{exp2} S_{s,h,q}[n,k]= h^{n-k} (1-q)^{k-n}
\sum_{\substack{ k\leq t_1 \leq t_3 \leq n}} (-1)^{t_3-t_2+t_1-k} \\ q^{(s - 1)\binom{t_1-k}{2} - (s-1)\binom{t_1}{2} + s\binom{t_3}{2}}
\binom{t_3}{t_2} \binom{t_2}{t_1} \pq{n}{t_3}_{q^s} \pq{t_1}{k}_{q^{s-1}}\,.
\end{multline*}
\begin{proof}Let $Q=1/q$, $v^*_i=Q^{si}-1$, $w^*_i=-(Q^{(s-1)i}-1)$, $c=1,d=-1$, $v_i=Q^{si},w_i=-Q^{(s-1)i}$. Observe that $v^*_i=(Q-1)[si]_Q$ and $w^*_i=(Q-1)[(s-1)i]_Q$. Hence, by \eqref{carl2},
\begin{equation*}
(Q-1)^{n-k} A^{\textnormal{\textbf {v}}^*,\textnormal{ \textbf{w}}^*}_{n,k} = \sum_{k\leq t_1 \leq t_2 \leq t_3 \leq n} \binom{t_3}{t_2} \binom{t_2}{t_1} (-1)^{t_3-t_2+t_1-k} Q^{s\binom{n-t3}{2}} \pq{n}{t_3}_{Q^s} \pq{t_1}{k}_{Q^{s-1}}\,.
\end{equation*}
Finally, we use the property $\pq{n}{k}_{1/p}=\pq{n}{k}_p p^{-\binom{n}{2}+\binom{k}{2}+\binom{n-k}{2}}$. \end{proof} \end{corollary}

\section{Explicit Formulas}\label{second}

Let $\textbf r=(r_1,r_2,\ldots,r_n), \textbf s=(s_1,s_2,\ldots,s_n)$. We will use $|\cdot|$ to denote the sum of the elements of a finite sequence. We now consider the normal ordering of strings of the form $H_{\textnormal{\textbf{r}},\textnormal{\textbf{s}}}=V^{r_n}U^{s_n}\cdots V^{r_2}U^{s_2}V^{r_1}U^{s_1}$ which can be written in the form
\begin{equation}\label {wnorm}
H_{\textnormal{\textbf{r}},\textnormal{\textbf{s}}} = \sum_
{k=s_1}^{|\textbf s|} S^{\textbf r, \textbf s}_{s,h,q}[k] V^{|\textbf r| - (|\textbf s|-k)(1-s)} U^k\,.
\end{equation}
In this section, we obtain explicit formulas for the numbers $S^{\textbf r, \textbf s}_{s,h,q}[k]$. 

Our first identity is a $q$-analogue of the following explicit formula for $S^{\textnormal{\textbf{r}}, \textnormal{\textbf{s}}}_{0,1,1}(k)=S^{\textbf r,\textbf s}(k)$ which was computed by El-Desouky, Caki\'c and Mansour \cite{DesCakMan} by repeated application of the Leibniz formula
\begin{equation}\label{des}
S^{\textbf r,\textbf s}(k) = \sum_{j_1+\cdots+j_{n-1}=s_1+\cdots+s_n-k}
\prod_{i=1}^{n-1} \binom{s_{i+1}}{j_i} \left( r_1+\cdots+ r_i - (j_1+\cdots+j_{i-1})\right)^{(\underline{j_i})}
\end{equation}
where $x^{(\underline{j})} = x(x-1)\cdots(x-(j-1))$ denotes the falling factorial. Our proof uses a rook theoretic interpretation of normal ordering. This interpretation was pointed out by Varvak \cite[Section 7]{Var} and is similar to the row creation rule introduced by Goldman and Haglund \cite{GolHag}.

The process of computing the normal ordering of a string $w$ can be considered as forming the collection of finite sequences of two operations, namely the conversion $UV$ to $qVU$ and $UV$ to $hV^s$, starting from the rightmost $UV$, such that successive applications of each element in the sequence produces a different string, until a string in normally ordered form is obtained. Let us denote the two operations by $\alpha$ and $\beta$, respectively. Applying a sequence containing $k$ $\beta$'s results to an expression $c V^{|\textbf r|-k(1-s)} U^{|\textbf s|-k}$. The sum of all such coefficients $c$ is then the coefficient of $V^{|\textbf r|-k(1-s)} U^{|\textbf s|-k}$ in the normal ordering of $H_{\textnormal{\textbf{r}},\textnormal{\textbf{s}}}$. The process we just described can be translated in terms of rook placements on a Ferrers board (or simply, board).

Let $U$ correspond to a horizontal step and $V$ a vertical step. Then $w$ outlines a board which we denote by $B(w)$. The conversion $\alpha$ corresponds to leaving a cell empty while the conversion $\beta$ corresponds to placing a rook on a cell such that the rook cancels all cells on top of it and divides the row to its left into $s$ rows. Note that if $s=0$, then a placement of a rook converts a row into zero rows, or equivalently, it cancels all the cells to its left. We denote a rook by marking a cell with ``$\bullet$'' and a canceled cell with ``$\times$''. Since the $\beta$'s convert the rightmost $UV$, the rooks are placed in some chosen columns from right to left. We can think of a cell lying in a divided row as containing \emph{subcells}, with each cell containing 1 subcell by default. A canceled cell is a assigned the weight 1 while a cell containing a rook is assigned the weight $h$. All other cells are assigned the weight $q^a$, where $a$ is the number of subcells in a cell. The weight of the rook placement is the product of the weight of the cells. Finally, the normally ordered string resulting from a particular rook placement is $\omega V^iU^{j}$, where $\omega$ is the weight of the rook placement, $i$ is the number of rows in the leftmost column plus $s-1$ if the leftmost column contains a rook, and $j$ is the number of columns not containing rooks. 
Note that the original model by Goldman and Haglund \cite{GolHag} involves creating $s$ new rows to the left of a cell containing a rook and then canceling the original row. The model we described involves dividing the row into $s$ rows, ie., creating $s-1$ new rows without canceling the original row. It is apparent that these models are equivalent, the only other notable difference being the orientation of the boards. Alternatively, we also say the placement of a rook in a cell adds $s-1$ subcells to every cell lying to its left. Let us call the rook placement rule we just described as the \emph{row creation rule}.

Figure \ref{board} shows a rook placement on $B(w)$, where $w=V^2U^3V^3U^2,s=2$. This placement corresponds to applying the sequence
$(\alpha,\alpha,\alpha,\alpha,\alpha,\beta,\alpha,\alpha,\beta,\alpha)$ from right to left. The string $h^2q^8 V^7U^3$ is produced.

\begin{figure}[htbp]
\begin{center}
\begin{tabular}{|c|c|c|c|c|}
\cline{1-5}
\phantom{h} & $\times$ & $\times$ & \multicolumn{1}{c}{} & \multicolumn{1}{c}{} \\

\cline{1-3}
& \multirow{2}{*}{$\bullet$} & \multirow{3}{*}{$\bullet$} & \multicolumn{1}{c}{} & \multicolumn{1}{c}{} \\

\cline{1-1}
& & & \multicolumn{1}{c}{} & \multicolumn{1}{c}{} \\

\cline{1-2}
& &  & \multicolumn{1}{c}{} & \multicolumn{1}{c}{} \\

\cline{1-3}
\multicolumn{1}{|c}{} & \multicolumn{1}{c}{} & \multicolumn{1}{c}{} & \multicolumn{1}{c}{} & \multicolumn{1}{c}{} \\

\multicolumn{1}{|c}{} & \multicolumn{1}{c}{} & \multicolumn{1}{c}{} & \multicolumn{1}{c}{} & \multicolumn{1}{c}{}
\end{tabular}
\caption{}
\end{center}
\label{board}
\end{figure}

Let $B$ be a board. Denote by $C_k(B;s)$ the collection of all placements of $k$ rooks on $B$ under the row creation rule. For a rook placement $\phi\in C_k(B,s)$, denote the weight of $\phi$ by $\omega(\phi)$. We define the \emph{rook number} $R_{s,h,q}[B,k]$ by
\begin{equation}\label{rooknumber1}
R_{s,h,q}[B,k] = \sum_{\phi\in C_k(B;s)} \omega(\phi)\,.
\end{equation}
One sees that the number of rooks completely determines the exponents of $V$ and $U$ in the word resulting from the rook placement. In particular, if $k$ rooks are placed, then $k$ columns are cancelled and $k(s-1)$ rows are added. Hence, we can write $H_{\textnormal{\textbf{r}},\textnormal{\textbf{s}}}$ as 
\begin{equation}\label{rooknorm}
H_{\textnormal{\textbf{r}},\textnormal{\textbf{s}}} = \sum_ {k=0}^{|\textbf s|-s_1} R_{s,h,q}[B(H_{\textnormal{\textbf{r}},\textnormal{\textbf{s}}}),k] V^{|\textbf r|-k(1-s)} U^{|\textbf s|-k}\,.
\end{equation}
Comparing this with \eqref{wnorm} gives $S^{\textnormal{\textbf{r}}, \textnormal{\textbf{s}}}_{s,h,q}[k] =
R_{s,h,q}[B(H_{\textnormal{\textbf{r}},\textnormal{\textbf{s}}}),|\textbf s|-k]$. It is important to note that $R_{s,h,q}[B(H_{\textnormal{\textbf{r}},\textnormal{\textbf{s}}}),|\textbf s|-k]=0$ when $k>|\textbf s|-s_1$ since the number of rooks cannot exceed the number of rows of positive length. This is also reflected in the fact that the exponent of $U$ in the normal ordering of $H_{\textnormal{\textbf{r}},\textnormal{\textbf{s}}}$ is between $s_1$ and $|\textbf s|$, or equivalently, that $S^{\textnormal{\textbf{r}}, \textnormal{\textbf{s}}}_{s,h,q}[k]=0$ when $k<s_1$ or $k>|\textbf s|$.

The following lemma will be used in deriving an analogue of \eqref{des}. An equivalent formula was derived in \cite{ManSch} using a different method.

\begin{lemma}\label{lemrep} We have
\begin{equation*}
U^{s'}V^{r'} = \sum_{j=0}^{s'} \left( h^jq^{r'(s'-j)} \pq{s'}{j}_{q^{s-1}} \prod_{i=0}^{j-1} [r'+i(s-1)]_q \right)V^{r'+j(s-1)} U^{s'-j}\,.
\end{equation*}
\begin{proof} Suppose that $j$ columns have been chosen where rooks will be placed. If the first rook is placed on the cell in the $i$th row, $1\leq i\leq r'$, then the cells below the rook will contribute a weight of $q^{i-1}$. As $i$ varies, a total weight of $h(1+q+\cdots+q^{r'-1})=h[r']_q$ will be contributed by all possible placements of the first rook. Since the placement of the first rook adds $s-1$ subcells to every cell to its left, the total weight contributed by all possible placement of the second rook is $[r'+(s-1)]_q$. Continuing this process with the other columns, we see that the weight contributed by all possible placements of $j$ rooks in the chosen columns is $h^j\prod_{i=0}^{j-1} [r+i(s-1)]_q$, and that this weight is the same for any choice of $j$ columns.

We now consider the weight contributed by the other columns in which no rooks are placed. For such a column, the weight is completely determined by the number of columns to its right that contains a rook, i.e., if there are $t$ columns to its right containing a rook, then the column will assume a weight of $q^{r'+t(s-1)}$. Note that $t$ varies from $0$ to $j$ and that for a given placement of $j$	 rooks, the weight contributed by all the columns containing no rooks is $q^{r't_0} q^{(r'+(s-1))t_1} q^{(r'+2(s-1))t_2} \cdots q^{(r'+j(s-1))t_j}$ for some $t_0+t_1+\cdots+t_j=s'-j$. Summing this up on all such possible collections $\{t_0,t_1,\ldots,t_i\}$, we have
\begin{align*}
\sum_{t_0+t_1+\cdots+t_j=s'-j} &q^{r't_0} q^{(r'+(s-1))t_1} q^{(r'+2(s-1))t_2} \cdots q^{(r'+j(s-1))t_j} \\
&= q^{r'(s'-j)} \sum_{t_0+t_1+\cdots+t_j=s'-j} q^{0(s-1)t_0} q^{1(s-1)t_1} q^{2(s-1)t_2} \cdots q^{j(s-1)t_j}\\ &= q^{r'(s'-j)} \pq{s'}{j}_{q^{s-1}}\,,
\end{align*}
by \eqref{tsetalt} with $v_i=0$ and $w_i=q^{i(s-1)}$. This proves the lemma. \end{proof}\end{lemma}

\begin{theorem}\label{theoremrep} The string $H_{\textnormal{\textbf{r}},\textnormal{\textbf{s}}}$ may be written as \begin{multline}\label{longsum}
H_{\textnormal{\textbf{r}},\textnormal{\textbf{s}}} = \sum_{j_1=0}^{s_2} \sum_{j_2=0}^{s_3} \cdots \sum_{j_{n-1}=0}^{s_n} \prod_{i=1}^{n-1}
h^{j_1+\cdots+j_{n-1}} \Gamma_{q,s} [j_i,r_1+\cdots+r_i+(j_1+\cdots+j_{i-1})(s-1),s_{i+1}] \\ V^{r_1+\cdots+r_n + (j_1+\cdots+j_{n-1})(s-1)}
U^{s_1+\cdots+s_n-(j_1+\cdots+j_{n-1})}\,.
\end{multline}
where
\begin{equation*} \Gamma_{q,s}[j,r',s'] = q^{r'(s'-j)} \pq{s'}{j}_{q^{s-1}}
\prod_{i=0}^{j-1} [r'+i(s-1)]_q\,.
\end{equation*}
Hence, the numbers $S^{\textnormal{\textbf{r}}, \textnormal{\textbf{s}}}_{s,h,q}[k]$ are given by
\begin{equation}\label{exp22}
S^{\textnormal{\textbf{r}}, \textnormal{\textbf{s}}}_{s,h,q}[k] =
h^{|\textnormal{\textbf{s}}|-k}\sum_{j_1+\cdots+j_{n-1}=s_1+\cdots+s_n-k}~\prod_{i=1}^{n-1 } \Gamma_{q,s}
[j_i,r_1+\cdots+r_i+(j_1+\cdots+j_{i-1})(s-1),s_{i+1}]\,.
\end{equation}
\begin{proof} Identity \eqref{longsum} is proved by repeated application of \eqref{lemrep} beginning from $U^{s_2}V^{r_1}$. Identity \eqref{exp22} follows by comparing the coefficient of $U^k$ in \eqref{longsum} and \eqref{wnorm}. \end{proof} \end{theorem}

\begin{corollary}\label{correp} The following explicit formula for $S_{s,h,q}[n,k]$ holds
\begin{equation}\label{exp3}
S_{s,h,q}[n,k] = h^{n-k} \sum_{j_1+\cdots+j_{n-1}=n-k}~\prod_{i=1}^{n-1} q^{(i+(j_1+\cdots+j_{i-1})(s-1))(1-j_i)} \pq{i+(j_1+\cdots+j_{i-1})(s-1)}{j_i}_q\,.
\end{equation}
\end{corollary}

Varvak's \cite{Var} use of rook factorization to obtain an explicit formula adapts readily in the case of $S^{\textnormal{\textbf{r}},\textnormal{\textbf{s}}}_{s,h,q}[k]$ after some modification. We will need the following analogues of the falling factorial and factorial: for $r\in\mathbb R,j\in\mathbb N$, define $ [r]^{(\underline{j})}_{q,1-s} = [r(1-s)]_q [(r-1)(1-s)]_q \cdots [(r-j+1)(1-s)]_q$ and for $n\in\mathbb N$, define $[n]_{q,1-s}!=[n]^{(\underline{n})}_{q,1-s}$.

\begin{theorem}\label{expexp}
Let $s\neq 1$. The coefficients $S^{\textnormal{\textbf{r}}, \textnormal{\textbf{s}}}_{s,h,q}[k]$ satisfy the explicit formula
\begin{equation}
S^{\textnormal{\textbf{r}}, \textnormal{\textbf{s}}}_{s,h,q}[k] = \frac{h^{|\textnormal{\textbf{s}}|-k}}{[k]_{q,1-s}!} \sum_{j=0}^k (-1)^{k-j} q^{\binom{k-j}{2}(1-s)} \pq{k}{j}_{q^{1-s}} \Omega^{\textnormal{\textbf r},\textnormal{\textbf s}}_{s,q}[j]
\end{equation}
where
\begin{equation*}
\Omega^{\textnormal{\textbf r},\textnormal{\textbf s}}_{s,q}[j]=\prod_{t=1}^{n} [j-(s_1+s_2+\cdots+s_{t-1})+(r_1+r_2+\cdots+r_{t-1})/(1-s)]_{q,1-s}^{(\underline{s_t})}\,.
\end{equation*}
\begin{proof} We use a representation of $V,U$ as linear operators whose action on the monomial $t^j$ is given by $Vt^j=t^{j+1}$ and $Ut^j=h[n]_qt^{j+s-1}$. One can verify that these operators satisfy $VU-qVU=hV^s$ and that $U^k t^{n(1-s)} =h^k [n]_{q,1-s}^{(\underline{k})} t^{(n-k)(1-s)}$. We then apply both sides of \eqref{wnorm} to $t^{x(1-s)}$. After letting $t=1$ to the resulting equation and using $[x]^{(\underline{k})}_{q,1-s}=[1-s]_q^k [x]_{q^{1-s}}[x-1]_{q^{1-s}} \cdots[x-k+1]_{q^{1-s}}$, we obtain
\begin{equation}
h^{|\textbf s|} \Omega^{\textnormal{\textbf r},\textnormal{\textbf s}}_{s,q}[x]= \sum_{k=s_1}^{|\textbf s|} h^{k} S^{\textnormal{\textbf{r}}, \textnormal{\textbf{s}}}_{s,h,q}[k][1-s]_q^k [x]_{q^{1-s}}[x-1]_{q^{1-s}} \cdots[x-k+1]_{q^{1-s}}\,.\label{sumed} 
\end{equation}

Let $E$ denote the shift operator $EP(x)=P(x+1)$ and $\Delta_{Q}^k$ the $k$-th $Q$-difference operator defined by $\Delta_{Q}^k=(E-1)(E-Q)\cdots (E-Q^{k-1})$. If $P(x)=\sum_{k}p_k [x]_Q [x-1]_Q \cdots [x-k+1]_Q$, then $p_k=\frac{1}{[k]_Q!} \Delta_Q^k P(x) \vert_{x=0}$. By the $q$-binomial theorem, $\Delta_{Q}^k = \sum_{j=0}^k (-1)^{j} Q^{\binom{j}{2}} \pq{k}{j}_{Q} E^{k-j}$. The result then follows by letting $Q=q^{1-s}$, $p_k=h^{k} S^{\textnormal{\textbf{r}}, \textnormal{\textbf{s}}}_{s,h,q}[k][1-s]_q^k$ and $P(x)=h^{|\textbf s|} \Omega^{\textnormal{\textbf r},\textnormal{\textbf s}}_{s,q}[x]$.
 \end{proof} \end{theorem}

\begin{corollary}\label{secexp} For $s \neq 1$ or $s\neq 0$, the numbers $S_{s,h,q}[n,k]$ has the following explicit formula
\begin{equation*}
S_{s,h,q}[n,k] = \frac{h^{n-k}[s]_q^n}{[k]_{q^{1-s}}![1-s]_q^k} \sum_{j=0}^{k} (-1)^{k-j} q^{\binom{k-j}{2}(1-s)}
\pq{k}{j}_{q^{1-s}} \prod_{t=1}^n [(j/s)+t-j-1]_{q^s}\,.
\end{equation*}
When $s=0$,
\begin{equation*}
S_{0,h,q}[n,k] = \frac{h^{n-k}}{[k]_q!} \sum_{j=0}^{k} (-1)^{k-j} q^{\binom{k-j}{2}} \pq{k}{j}_{q} [j]_q^n\,.
\end{equation*} \end{corollary}

\section{$q$-Bell numbers}\label{third}

The Bell polynomial $B_n(x)$ is defined as the sum $\sum_{k=0}^n S(n,k) x^k$ while the Bell number $B_n$ is given by $B_n(1)$. Analogously, we define the generalized $q$-Bell polynomials $B^{\textbf r,\textbf s}_{s,h,q}[x]$ and generalized $q$-Bell numbers $B^{\textbf r,\textbf s}_{s,h,q}$ as
\begin{equation*}
B^{\textbf r,\textbf s}_{s,h,q}[x] = \sum_{k=s_1}^{|\textbf s|} S^{\textnormal{\textbf{r}}, \textnormal{\textbf{s}}}_{s,h,q}[k] x^{k}, \hspace{0.25in} B^{\textbf r,\textbf s}_{s,h,q}=B^{\textbf r,\textbf s}_{s,h,q}[1]\,.
\end{equation*}
If $\textbf r=(1,1,\ldots,1),\textbf s=(1,1,\ldots,1)$, we define
\begin{equation*} B_{s,h,q}[n;x] = B^{\textbf r,\textbf s}_{s,h,q}[x], \hspace{0.25in} B_{s,h,q}[n]=B_{s,h,q}[n;1]\,.
\end{equation*}
For all the other particular cases, we apply the same notational convention in Table \ref{table}. The numbers $B_{s,h,q}[n;x]$ and $B_{s,h,q}[n]$ reduce to the usual Bell polynomial $B(n;x)$ and Bell number $B(n)$, respectively, when $q=1,h=1,s=0$. The Bell polynomial $B(n;x)$ have the following expression as an infinite series known as the \emph{Dobinsky formula} (see \cite{Xu}) 
\begin{equation*}
B(n;x) = \frac{1}{e^x} \sum_{j=0}^{\infty} j^n \frac{x^j}{j!}\,.
\end{equation*}

The corollary that follows gives the Dobinsky-type formula for $B^{\textbf r,\textbf s}_{s,h,q}[x]$.

\begin{corollary}\label{dobcor} Let $s\neq 1$ and $\Omega^{\textnormal{\textbf r},\textnormal{\textbf s}}_{s,q}[j]$ be as in Theorem \ref{expexp}. Then,
\begin{equation}\label{dobdob}
B^{\textbf r,\textbf s}_{s,h,q}[x]=\left(\sum_{j=0}^{\infty} h^{|\textbf s|-j} (-1)^j q^{\binom{j}{2}(1-s)}
\frac{x^{j}}{[j]_{q,1-s}!} \right)\left(\sum_{j=0}^{\infty} \Omega^{\textnormal{\textbf r},\textnormal{\textbf s}}_{s,q}[j]
\frac{x^{j}}{h^j[j]_{q,1-s}!}\right)\,.
\end{equation}
\begin{proof} By the property $S^{\textnormal{\textbf{r}}, \textnormal{\textbf{s}}}_{s,h,q}[k]=0$ when $|\textbf s|<k<s_1$ and by Theorem \ref{expexp}\,, 
\begin{align*}
B^{\textbf r,\textbf s}_{s,h,q}[hx] &= \sum_{k=s_1}^{|\textbf s|} S^{\textnormal{\textbf{r}}, \textnormal{\textbf{s}}}_{s,h,q}[k] h^k x^{k} \\
&= \sum_{k=0}^{\infty} S^{\textnormal{\textbf{r}}, \textnormal{\textbf{s}}}_{s,h,q}[k] h^k x^{k} \\
&= \sum_{k=0}^{\infty} \sum_{j=0}^k \frac{h^{|\textnormal{\textbf{s}}|}}{[k]_{q,1-s}!} (-1)^{k-j} q^{\binom{k-j}{2}(1-s)} \pq{k}{j}_{q^{1-s}} \Omega^{\textnormal{\textbf r},\textnormal{\textbf s}}_{s,q}[j] x^k\,.
\end{align*}
Using $\pq{k}{j}_{q^{1-s}}=\frac{[k]_{q^{1-s}}!}{[j]_{q^{1-s}}![k-j]_{q^{1-s}}!}=\frac{[k]_{q,1-s}!}{[j]_{q,1-s}![k-j]_{q,1-s}!}$ and the Cauchy product rule,
\begin{align*}
B^{\textbf r,\textbf s}_{s,h,q}[hx] &= \sum_{k=0}^{\infty} \sum_{j=0}^k \left(\frac{h^{|\textnormal{\textbf{s}}|}
(-1)^{k-j} q^{\binom{k-j}{2}(1-s)}}{[k-j]_{q,1-s}!}\right)\left( \frac{\Omega^{\textnormal{\textbf r},\textnormal{\textbf
s}}_{s,q}[j]}{[j]_{q,1-s}!}\right) x^k \\
&= \left(\sum_{j=0}^{\infty} h^{|\textbf s|} (-1)^j q^{\binom{j}{2}(1-s)} \frac{x^{j}}{[j]_{q,1-s}!}
\right)\left(\sum_{j=0}^{\infty} \Omega^{\textnormal{\textbf r},\textnormal{\textbf s}}_{s,q}[j] \frac{x^{j}}{[j]_{q,1-s}!}\right)\,.
\end{align*} \end{proof} \end{corollary}

Our goal in the remainder of this section is to obtain $q$-analogues of the following Bell number identities derived by Mansour, Schork and Shattuck \cite[Theorems 4.4 and 5.3]{ManSchSha1}:
\begin{align}
B_{s,h}(n) &= \sum_{r=0}^{n-1} h^{n-r-1} \binom{n-1}{r} B_{s,h}(r) \prod_{i=0}^{n-r-2} (1+si)\label{man1}\\
B_{s,h}(n+m) &= \sum_{r=0}^{n} \sum_{j=0}^m h^{n-r} \binom{n}{r} S(m,j)B_{s,h}(r) \prod_{i=0}^{n-r-1} (j(1-s)+sm+si)\label{man2}\,.
\end{align}
One verifies that \eqref{man2} reduces to \eqref{man1} when $n$ is replaced with $n-1$ and $m$ with $1$. The corresponding identity for the classical Bell numbers was first derived by Spivey \cite{Spi} and is given by
\begin{equation}\label{spivey}
B(n+m) = \sum_{r=0}^{n} \sum_{j=0}^m \binom{n}{r} S(m,j) B(r) j^{n-r}\,.
\end{equation}

In addition to \eqref{man2}, generalizations of the identity \eqref{spivey} have been proved using different methods (see \cite{Xu2} and the references therein). The approach we present here uses the rook model in Section \ref{second}.

Recall that $C_k(B;s)$ is the collection of all placements of $k$ rooks in the board $B$. Denote by $J_n$ the board outlined by the string $(VU)^n$. In the proofs that follow, we call a cell a \emph{bottom cell} if it is the bottommost cell in a column.

\begin{lemma}\label{col} Let $s\in\mathbb N$ and $\phi\in C_k(J_n;s)$. Then, there exists a unique (possibly empty) collection $\mathcal C$ of columns in $\phi$ such that if $|\mathcal C|=m+1$, then (a) each of these columns has a rook in the bottom $1,1+s,1+2s,\ldots,1+ms$ subcells and (b) every column not in $\mathcal C$ contains at least $1+st$ uncanceled subcells not containing a rook, where $t$ is the number of columns in $\mathcal C$ to the right of that column.
\begin{proof} Let $\phi\in C_k(J_n;s)$. The set $\mathcal C$ may be obtained as follows. Let $c_1$ be the first column of $\phi$ from the right containing a rook in the bottom cell. If $c_1$ does not exist, then $\mathcal C=\varnothing$ and $m=-1$. If $c_1$ exists, then all columns to the right of $c_1$ have at least 1 uncanceled cell not containing a rook. Let $c_2$ be the first column to the left of $c_1$ containing a rook in the bottom $1+s$ subcells. If $c_2$ does not exist, then all columns to the left of $c_1$ contain at least $1+s$ uncanceled subcells not containing a rook. Hence, $\mathcal C=\{c_1\}$ satisfies (a) and (b) and $m=0$. Otherwise, if such a column $c_2$ exists, then all columns to the right of $c_2$ and to the left of $c_1$ contain $1+s$ uncanceled subcells not containing a rook. Let $c_3$ be the first column to the left of $c_2$ containing a rook in the bottom $1+2s$ subcells. If $c_3$, does not exist, then $\mathcal C=\{c_1,c_2\}$ satisfies (a) and (b) and $m=1$. We repeat the process with the succeeding columns as long as needed until all the elements of $\mathcal C$ are determined. This process shows both the existence and uniqueness of $\mathcal C$, which proves the lemma. \end{proof} \end{lemma}

\begin{theorem}\label{firstiden} Let $n,k\in\mathbb N$. Then
\begin{equation}\label{ssum}
S_{s,h,q}[n,k] = \sum_{r=k-1}^{n-1} h^{n-r-1} q^r \pq{n-1}{r}_{q^s} S_{s,h,q}[r,k-1] \prod_{i=0}^{n-r-2} [1+si]_q\,.
\end{equation}
Furthermore, the numbers $B_{s,h,q}[n]$ are given by
\begin{equation}\label{bsum} B_{s,h,q}[n] = \sum_{r=0}^{n-1} h^{n-r-1} q^r \pq{n-1}{r}_{q^s} B_{s,h,q}[r] \prod_{i=0}^{n-r-2} [1+si]_q\,.
\end{equation}
\begin{proof} We first prove \eqref{ssum}. The number $S_{s,h,q}[n,k]$ equals the sum of the weights of all rook placements in $C_{n-k}(J_n;s)$. Let $R_{n-k}(J_n;s;n-r-2)$ be the subset of $C_{n-k}(J_n;s)$ consisting of rook placements that satisfy conditions (a) and (b) in Lemma \ref{col} with $m=n-r-2$. Then, the same lemma implies that the collection $\{R_{n-k}(J_n;s;n-r-2)\,|\,k-1\leq r< n-1\}$ forms a partition of $C_{n-k}(J_n;s)$. We want to show that
\begin{equation*}
\sum_{\phi\in R_{n-k}(J_n;s;n-r-2)} \omega(\phi) = h^{n-r-1} q^r \pq{n-1}{r}_{q^s} S_{s,h,q}[r,k-1] \prod_{i=0}^{n-r-2} [1+si]_q\,,
\end{equation*}
from which \eqref{ssum} follows.

For $\phi\in R_{n-k}(J_n;s;n-r-2)$, denote its set of columns satisfying the conditions in Lemma \ref{col} by $\mathcal C_{\phi}$. Let
$\lambda_{\mathcal C_{\phi}}$ be the cells of $\phi$ consisting of its cells of the columns $\mathcal C_{\phi}$ and the bottom $st+1$ cells of the other columns, where $t$ is the number of columns from $\mathcal C_{\phi}$ that lie to the right of one such column. Also, let $\phi-\lambda_{\mathcal C_{\phi}}$ be the cells of $\phi$ not in $\lambda_{\mathcal C_{\phi}}$. One sees that the cells in $\phi-\lambda_{\mathcal C_{\phi}}$ form a rook placement in $C_{r-k+1}(J_{r};s)$. We can therefore write every rook placement $\phi\in
R_{n-k}(J_n;s;n-r-2)$ uniquely as a pair $(\lambda_{\mathcal C},\rho)$, for some set of columns $\mathcal C$ satisfying Lemma \ref{col} with $m=n-r-2$, and some rook placement $\rho$ in $C_{r-k+1}(J_{r};s)$. The sum of the weights of rook placements in $C_{r-k+1}(J_{r};s)$ is $S_{s,h,q}[r,k-1]$. We now compute the sum of the weights of the cells in $\lambda_{\mathcal C}$ over all such possible set of columns $\mathcal C$, which we denote by $L_{n-r-2}$.

Clearly, the cells in $\lambda_{\mathcal C}$ which contain rooks contribute a weight of $h^{n-r-1} \prod_{i=0}^{n-r-2} [1+si]_q$. In addition, $\lambda_{\mathcal C}$ contains $r$ bottom cells, which collectively contribute a weight of $q^r$. The weight contributed by the remaining cells depends on the location of the rooks. To get a better picture of how the contribution by the remaining cells varies, let us distribute the cells of $\lambda_{\mathcal C}$ so that the $s$ subcells lie to the left of each rook. We illustrate this in Figure \ref{vary} for $n=9$ and $n-r-2=3$. Here, the second column of $\lambda_{\mathcal C}$ from the left lies to the left of three columns of $\lambda_{\mathcal C}$ containing rooks. Hence, this column contains $1+3s$ bottom cells. We moved each of the three $s$ subcells (indicated by marking a cell with an $s$) so that they lie to the left of each of the three rooks. Going back to the general case, let $t_1,t_2,\ldots,t_{n-r-1}$ be the number of cells marked $s$ in the rows containing the rooks in $\lambda_{\mathcal C}$ starting from bottommost row. Then the $t_i$'s satisfy $0\leq t_1\leq t_2\leq \ldots \leq t_{n-r-1}\leq r$. Hence, if $\omega(\lambda_{\mathcal C})$ denotes the product of the weights of the cells in $\lambda_{\mathcal C}$, we have
\begin{align*}
\sum_{\mathcal C\in L_{n-r-2}} \omega(\lambda_{\mathcal C}) &= h^{n-r-1} q^r \prod_{i=0}^{n-r-2} [1+si]_q \sum_{0\leq t_1\leq t_2\leq \ldots \leq t_{n-r-1}\leq r} q^{st_1} q^{st_2} \cdots q^{st_{n-r-1}} \\
&= h^{n-r-1} q^r \prod_{i=0}^{n-r-2} [1+si]_q \pq{n-1}{r}_{q^s}\,.
\end{align*}
where the second equality follows from \eqref{tset} with $v_i=0,w_i=q^{is}$. This proves \eqref{ssum}.

To prove \eqref{bsum}, we take the sum of both sides of \eqref{ssum} over all $0\leq k\leq n$. \end{proof} \end{theorem} 

\begin{figure}[htbp]
\begin{center}
\begin{tabular}{|c|c|c|c|c|c|c|c|c}
\cline{1-9}
$\times$ & $s$ & $s$ & $s$ & $\times$ & $s$ & $\times$ & $\bullet$ &  \multicolumn{1}{c}{\hphantom{x}} \\
\cline{1-8}
$\times$ & $s$ & $s$ & $s$ & $\times$ & $s$ & $\bullet$ & \multicolumn{1}{c}{} &  \multicolumn{1}{c}{} \\
\cline{1-7}
$\times$ & & & & $\times$ & & \multicolumn{1}{c}{} & \multicolumn{1}{c}{} &  \multicolumn{1}{c}{}\\
\cline{1-6}
$\times$ & $s$ & $s$ & $s$ & $\bullet$ & \multicolumn{1}{c}{} & \multicolumn{1}{c}{} & \multicolumn{1}{c}{} &  \multicolumn{1}{c}{} \\
\cline{1-5}
$\times$ &  &  &  & \multicolumn{1}{c}{} & \multicolumn{1}{c}{} & \multicolumn{1}{c}{} & \multicolumn{1}{c}{} &  \multicolumn{1}{c}{} \\
\cline{1-4}
$\times$ &  &  & \multicolumn{1}{c}{} & \multicolumn{1}{c}{} & \multicolumn{1}{c}{} & \multicolumn{1}{c}{} & \multicolumn{1}{c}{} &  \multicolumn{1}{c}{} \\
\cline{1-3}
$\times$ &  & \multicolumn{1}{c}{} & \multicolumn{1}{c}{} & \multicolumn{1}{c}{} & \multicolumn{1}{c}{} & \multicolumn{1}{c}{} & \multicolumn{1}{c}{} &  \multicolumn{1}{c}{} \\
\cline{1-2}
$\bullet$ & \multicolumn{1}{c}{}  & \multicolumn{1}{c}{} & \multicolumn{1}{c}{} & \multicolumn{1}{c}{} & \multicolumn{1}{c}{} & \multicolumn{1}{c}{} & \multicolumn{1}{c}{} &  \multicolumn{1}{c}{} \\
\cline{1-1}
\multicolumn{1}{|c}{} & \multicolumn{1}{c}{}  & \multicolumn{1}{c}{} & \multicolumn{1}{c}{} & \multicolumn{1}{c}{} & \multicolumn{1}{c}{} & \multicolumn{1}{c}{} & \multicolumn{1}{c}{} &  \multicolumn{1}{c}{} \\
\end{tabular}
\caption{}
\end{center}
\label{vary}
\end{figure}

\begin{lemma} \label{col2} Denote by $J_{n\oplus \alpha}$ the board outlined by $(VU)^{n} V^{\alpha}$. Let $s\in\mathbb N$ and $\phi\in C_{k}(J_{n\oplus \alpha};s)$. Then, there exists a unique (possibly empty) collection $\mathcal C$ of columns in $\phi$ such that if $|\mathcal C|=m+1$, then (a) each of these columns has a rook in the bottom $\alpha,\alpha+s,\ldots,\alpha + sm$ subcells and (b) every column not in $\mathcal C$ contains at least $\alpha + st $ uncanceled subcells not containing a rook, where $t$ is the number of columns in $\mathcal C$ to the right of that column.
\begin{proof} The proof is similar to that of Lemma \ref{col}. \end{proof} \end{lemma}

\begin{theorem}\label{seciden} Let $n,m,k\in\mathbb N$. We have
\begin{equation}\label{nmk}
S_{s,h,q}[n+m,k] = \sum_{r=0}^n \sum_{j=0}^m h^{n-r}
q^{r(j(1-s)+sm)} \pq{n}{r}_{q^s} S_{s,h,q}[m,j] S_{s,h,q}[r,k-j] \prod_{i=0}^{n-r-1} [j(1-s)+sm + si]_q \,.
\end{equation}
Moreover, the numbers $B_{s,h,q}[n+m]$ are given by
\begin{equation}\label{nm}
B_{s,h,q}[n+m] = \sum_{r=0}^n \sum_{j=0}^m h^{n-r} q^{r(j(1-s)+sm)} \pq{n}{r}_{q^s} S_{s,h,q}[m,j] B_{s,h,q}[r] \prod_{i=0}^{n-r-1} [j(1-s)+sm + si]_q\,.
\end{equation}
\begin{proof} We first prove $\eqref{nmk}$. The number $S_{s,h,q}[n+m,k]$ equals the total weight of all rooks placements in $C_{n+m-k}(J_{n+m};s)$. The rooks may be placed as follows: Number the columns from right to left. For some $j$ such that $0\leq j\leq m$, place $m-j$ rooks in columns $2,\ldots,m$ and the remaining $n+j-k$ rooks in columns $m+1,\ldots, n$. The total weight of all placements of $m-j$ rooks in columns $2,\ldots,m$ is $S_{s,h,q}[m,j]$. As a consequence of placing $m-j$ rooks, each of columns $m+1,\ldots,n$ have ${s(m-j)+j = j(1-s) + sm} $ subcells in their first $m$ cells from the top. Hence, these columns form the board $J_{n\oplus\alpha}$, with $\alpha=j(1-s) + sm$. Using Lemma \ref{col2}, we can form a suitable partition of $C_k(J_{n\oplus\alpha};s)$ and proceed as in the proof Identity \eqref{ssum} in Theorem \ref{firstiden}. We leave the details to the reader.

To obtain \eqref{nm}, take the sum of both sides of \eqref{nmk} over all $0\leq k\leq n+m$. \end{proof} \end{theorem} 

\begin{corollary}\label{thirdiden} The recurrence relation for the generalized $q$-Bell polynomials is given by
\begin{equation*}
B_{q}[n+m;x] = \sum_{r=0}^n \sum_{j=0}^m h^{n-r} q^{r(j(1-s)+sm)} \pq{n}{r}_{q^s} S_{s,h,q}[m,j] B_{q}[r;x]x^j \prod_{i=0}^{n-r-1} [j(1-s)+sm + si]_q\,.
\end{equation*}
In particular,
\begin{equation*}
B_{s,h,q}[n;x] = \sum_{r=0}^{n-1} h^{n-r-1} q^r \pq{n-1}{r}_{q^s} B_{s,h,q}[r;x] x \prod_{i=0}^{n-r-2} [1+si]_q\,.
\end{equation*}
Moreover, a $q$-analogue of Spivey's identity \eqref{spivey} is
\begin{equation*} B_{q}[n+m] = \sum_{r=0}^n \sum_{j=0}^m q^{rj} \binom{n}{r} S_{s,h,q}[m,j] B_{q}[r] [j]_q^{n-r}
\end{equation*}
where we take $0^0=1$. \end{corollary}

\section{The numbers  $S^{\textnormal{\textbf{r}}, \textnormal{\textbf{s}}}_{s,h,q}[k]$ and $S_{s,h,q}[n,k]$ for $s\in\mathbb R$}\label{fourth}

Goldman and Haglund introduced, in the latter part of their paper \cite{GolHag}, a rook placement rule which is defined for $s\in\mathbb R$, such that the corresponding rook numbers agree with the rook numbers under the row creation rule when $s\in\mathbb N$. We give an presentation of this rule using the concept of \emph{pre-weights}. As before, we place the rooks in some chosen columns from right to left such that at most one rook occupies each column. We say that every cell has a default pre-weight of 1 and a placement of a rook adds a pre-weight of $s-1$ to every cell to its left. Every cell lying above a rook is assigned the weight 1. If a cell does not lie above a rook and has pre-weight $p$, then it is assigned the weight $h[p]_q$ if it contains a rook, and $q^p$ if otherwise. Let us call the rook placement rule we just described as the \emph{pre-weight rule}. We denote the weight of a rook placement $\psi$ by $\omega(\psi)$, which is defined as the product of the weights of the cells. Given a board $B$, we denote by $C^*_k(B;s)$ the collection of all placement of $k$ rooks on $B$ under the pre-weight rule. When $s\in\mathbb N$, the pre-weight of a cell is exactly the number of its subcells. In this case, we have
\begin{equation*}
R_{s,h,q}[B(w),k] = \sum_{\phi\in C_k(B;s)} \omega(\phi) = \sum_{\psi\in C^*_k(B;s)} \omega(\psi)\,.
\end{equation*}
We can therefore define $R[B(w),k]$ as
\begin{equation*}
R_{s,h,q}[B(w),k] = \sum_{\psi\in C^*_k(B;s)} \omega(\psi)\,.
\end{equation*}
when $s\in\mathbb R$ without ambiguity. Using this definition, we can extend numbers the $S^{\textbf r,\textbf s}_{s,h,q}[k]$ for $s\in\mathbb R$ by defining them as $S^{\textnormal{\textbf{r}}, \textnormal{\textbf{s}}}_{s,h,q}[k] =
R_{s,h,q}[B(H_{\textnormal{\textbf{r}},\textnormal{\textbf{s}}}),|\textbf s|-k]$. Analogously, we define $S_{s,h,q}[n,k]=R_{s,h,q}[B((VU)^n),|\textbf s|-k]$. It can be shown that under this definition, $S_{s,h,q}[n,k]$ satisfies the same recursion\eqref{recrel}. Since the identities for $S_{s,h,q}[n,k]$ in Section \ref{first} were derived using this recursion, all of them also hold when $s\in\mathbb R$. Note that the numbers $S^{\textnormal{\textbf{r}}, \textnormal{\textbf{s}}}_{s,h,q}[n,k]$ lose their interpretation as normal ordering coefficients when $s$ is not a nonnegative integer. 

The proof of Lemma \ref{lemrep}, and hence, of Theorem \ref{theoremrep} (explicit formula for $S^{\textnormal{\textbf{r}}, \textnormal{\textbf{s}}}_{s,h,q}[k]$) and Corollary \ref{correp} (explicit formula for $S_{s,h,q}[n,k]$) extends readily under the pre-weight rule. We leave the details to the reader. On the other hand, the proof of Theorem \ref{expexp} (explicit formula for $S^{\textnormal{\textbf{r}}, \textnormal{\textbf{s}}}_{s,h,q}[k]$), and hence, of Corollary \ref{secexp} (explicit formula for $S_{s,h,q}[n,k]$) and Corollary \ref{dobcor} (Dobinsky-type formula), uses the rook factorization \eqref{sumed}, which was computed using the normal ordering interpretation. The result \cite[Theorem 7.1]{GolHag} furnishes the needed rook factorization for $h=1$, which will allow us to extend the said theorem and its corollaries. This result can be extended to arbitrary $h$ without much difficulty.

To prove Theorem \ref{firstiden} (recursion for $S_{s,h,q}[n,k]$ and $B_{s,h,q}[n]$), Theorem \ref{seciden} (recursion for $S_{s,h,q}[n+m,k]$ and $B_{s,h,q}[n+m]$) and Corollary \ref{thirdiden} for the case $s\in\mathbb R$, we impose a modification in the assignment of pre-weights when the board is $J_n$. First, as usual, choose columns where we will place rooks. The first rook adds a pre-weight of $s-1$ to each cell that lies above the bottom cell in every column to its left, the second rook adds a pre-weight of $s-1$ to the cell that lies above the cell in every column to its left which has been added a pre-weight of $s-1$ by the first rook, etc. We denote by $C^{\#}_k(J_n;s)$ the set of all rook placements on $J_n$ under this rule, which we call the \emph{modified pre-weight rule}. This modification preserves the sum of the weights of all rook placements given a choice of columns where rooks are to be placed. Hence, 
\begin{equation}\label{refw}
\sum_{\psi\in C^*_k(J_n;s)} \omega(\psi) = \sum_{\rho\in C^{\#}_k(J_n;s)} \omega(\rho)\,.
\end{equation}

\begin{proof}[Proof of Theorems \ref{firstiden} and \ref{seciden} for $s\in\mathbb R$] We only describe some portions of the necessary changes in the proofs of Theorems \ref{firstiden} and \ref{seciden} for $s\in\mathbb R$. We leave the rest of the details to the reader.

For Theorem \ref{firstiden}, we will need the following modification of Lemma \ref{col}: \emph{Let $\phi\in C^{\#}_k(J_n;s)$. Then, there exists a unique (possibly empty) collection $\mathcal C$ of columns in $\phi$ such that if $|\mathcal C|=m+1$, then (a) each of these columns has a rook in the bottom $1,2,3,\ldots,1+m$ cells and (b) every column not in $\mathcal C$ contains at least $1+t$ uncanceled cells not containing a rook, where $t$ is the number of columns in $\mathcal C$ to the right of that column.} 

For $1\leq j\leq 1+m$, the bottom $j$ cells referred to in (a) have a combined pre-weight of $1+s(j-1)$, and hence, the sum of the weights of all possible placement of rooks on these $j$ cells is $h[1+s(j-1)]_q$. On the other hand, the $1+t$ uncanceled cells in (b) have a combined pre-weight of $1+st$ and hence, contributes a weight of $q^{1+st}$.

We now turn to Theorem \ref{seciden}. As in the proof of the case $s\in\mathbb N$, we place $m-j$ rooks in columns $2,\ldots,m$ of $J_{n+m}$. As a consequence, the first $m$ cells from the top of columns $m+1,\ldots,n$ have a total pre-weight of $j(1-s)+sm$. Let $B$ be the board consisting of columns $m+1,\ldots,n$ of $J_{n+m}$ after  $m-j$ rooks were placed. 

It can be verified that the relation \eqref{refw} also holds when some of the cells of $J_n$ have default pre-weights other than 1. Let $J'_{n,\alpha}$ be the board $(VU)^nV$ such that the bottom cells have a default pre-weight of $\alpha$. One sees that sum of the pre-weights of the $t$-th column of $B$ and the sum of the pre-weights of the $t$-th column of $J'_{n,\alpha}$ with $\alpha=j(1-s)+sm$ are equal. This implies that the sum of the weights of the rook placements in $C^{\#}_k(B;s)$ equals that of $C^{\#}_k(J'_{n,\alpha};s)$. The rest of the proof uses a generalization of Lemma \ref{col2}, which is the same statement as the modification of Lemma \ref{col}, except that we replace $C^{\#}_k(J_n;s)$ with $C^{\#}_k(J'_{n,\alpha};s)$.
\end{proof}

\section{Conclusion} We have obtained a number of identities for $S_{s,h,q}[n,k]$, including an orthogonality relation, recurrence formulas and expressions involving other $q$-analogues. In particular, it was shown that $S_{s,h}(n,k)$ can be written in terms of the classical Stirling numbers and $S_{s,h,q}[n,k]$ in terms of $q$-Stirling numbers. Explicit formulas for $S^{\textnormal{\textbf{r}},\textnormal{\textbf{s}}}_{s,h,q}[k]$, and consequently, $S_{s,h,q}[n,k]$, were also given. Recurrence relations for $S_{s,h,q}[n,k]$ and $B_{s,h,q}[n]$ have also been proved using rook placements. As a consequence, we have also provided alternative proofs of the original identities in \cite{ManSchSha1} which were proved using weighted Laguerre configurations. It is possible that the identities for $S_{s,h,q}[n,k]$ in Section \ref{first} also have rook theoretic proofs. When viewed as normal ordering coefficients, the numbers $S^{\textnormal{\textbf{r}},\textnormal{\textbf{s}}}_{s,h,q}[k]$ are defined only for $s\in\mathbb N$. These numbers, as well as the identities they possess which were derived in the earlier sections, were extended to the case $s\in\mathbb R$ using a modified rook model.

We have not considered the corresponding boson operators and their action on coherent states that arise as a consequence of the generalized commutation relation. We note that the $q$-deformed case where $h=1,s=0$ was studied by Schork \cite{Sch} while Blasiak \cite{Bla} studied the undeformed case and obtained the normal ordering of a more general class of expressions involving the boson operators.

\section*{Acknowledgment}
K.J.M. Gonzales would like to thank the Department of Science and Technology (DOST) through the ASTHRD Program for financial support during his stay at the University of the Philippines - Diliman.

\bibliographystyle{amsplain}

\end{document}